\documentclass[12pt,oneside]{article}
\usepackage{amsmath,amssymb,amsfonts,amsthm}
\usepackage{color}
\newcommand{\T}{{\cal T}}

\newcommand{\Real}{\mathbb R}

\newcommand{\set}[1]{\left\{#1\right\}}

\def\Section#1{\vspace{30truept}\addtocounter{section}{1}\setcounter{thm}{0}\setcounter{equation}{0}
{\noindent\Large\bf\arabic{section}.~~#1}\par \vspace{12pt}}
\newtheorem{thm}{Theorem}[section]
\newtheorem{cor}[thm]{Corollary}
\newtheorem{lem}[thm]{Lemma}
\newtheorem{prop}[thm]{Proposition}
\newtheorem{defn}[thm]{Definition}

\newtheorem{rem}[thm]{Remark}

\numberwithin{equation}{section}
\begin{document}
\title{\bf Characterization of Finsler Spaces\\ of Scalar Curvature}%\footnote{ArXiv Number: }}
\author{\bf{ Nabil L. Youssef$^{\,1}$ and A. Soleiman$^{2}$}}
\date{}
%\thanks{\it Department of Mathematics, etc}
%\pagestyle{fancy}
             % End of preamble and beginning of text.
\maketitle                     % Produces the title.
\vspace{-1.16cm}
\begin{center}
{$^{1}$Department of Mathematics, Faculty of Science, Cairo
University, Giza, Egypt.\\ nlyoussef@sci.cu.edu.eg,\, nlyoussef2003@yahoo.fr}
\end{center}

\begin{center}
{$^{2}$Department of Mathematics, Faculty of Science, Benha
University, Benha,
 Egypt.\\ amr.hassan@fsci.bu.edu.eg,\, amrsoleiman@yahoo.com}
\end{center}

\vspace{0.7cm} \maketitle
\smallskip

%%%%%%%%%%%%%%%%%%%%%%%%%%%%%%%%%%%%%%%%%%%%%%%%%%%%%%%%%%%%%%% Abstract %%%%%%%%%%%%%%%%%%%%%%%%%%%%%%%%%%%%%%%%%%%%%%%%%%%%%%%%%%%

\noindent{\bf Abstract.}
The aim of the present paper is to
provide an intrinsic investigation of two special
Finsler spaces whose defining properties are related to Berwald connection, namely,
Finsler space of scalar curvature and of constant curvature. Some
characterizations of a Finsler space of scalar
curvature are proved. Necessary and sufficient conditions
under which a Finsler space of scalar curvature reduces to a
Finsler space of constant curvature are investigated.

\vspace{3pt}
\medskip\noindent{\bf Keywords.\/}\, Berwald connection, Deviation tensor, indicatory tensor,
 Finsler manifold of scalar curvature, Finsler manifold of constant curvature.

\vspace{3pt}
\medskip\noindent{\bf MSC 2010.\/} 53C60,
53B40
\bigskip
\vspace{8pt}
%%%%%%%%%%%%%%%%%%%%%%%%%%%%%%%%%%%%%%%%%%%%%%%%%%%%%%%%% Introduction %%%%%%%%%%%%%%%%%%%%%%%%%%%%%%%%%%%%%%%%%%%

\vspace{30truept}\centerline{\Large\bf{Introduction}}\vspace{12pt}
\par
Special Finsler spaces are investigated locally by many
authors (cf., for example,\linebreak \cite{r75}-\cite{r34}, \cite{r42}, \cite{sak.}, \cite{yos.1}).  On
the other hand, the global or intrinsic investigation of such spaces is rare in the
literature. Some contributions in this direction are found in
 \cite{r48} and \cite{r86}.
\par
In the present paper, we treat intrinsically two types of special Finsler spaces: Finsler space of scalar curvature and
Finsler space of constant curvature. Some characterizations of Finsler spaces of scalar curvature are proved. Necessary and
sufficient conditions for Finsler space of scalar curvature to reduces to a Finsler space of constant curvature are investigated.
\par
It should be noted that some important results of \cite{r42},
\cite{sak.} and \cite{yos.1} are
retrieved from the obtained global results, when localized.
\newpage

%%%%%%%%%%%%%%%%%%%%%%%%%%%%%%%%%%%%%%%%%%%%% SECTION 1. Notation and Preliminaries %%%%%%%%%%%%%%%%%%%%%%%%%%%%

\Section{Notation and Preliminaries}
In this section, we give a brief account of the basic concepts
 of the pullback approach to intrinsic Finsler geometry necessary for this work. For more
 details, we refer to \cite{r58}, \cite{r93} and~\,\cite{r48}.
The following notation will be used throughout this paper:\\
$M$: a real differentiable manifold of finite dimension $n$ and class $C^{\infty}$,\\
$\mathfrak{F}(M)$: the $\Real$-algebra of differentiable functions on $M$,\\
$\pi: \T M\longrightarrow M$: the subbundle of nonzero vectors tangent to $M$,\\
$P:\pi^{-1}(TM)\longrightarrow \T M$ : the pullback of the tangent bundle $TM$ by $\pi$,\\
$\mathfrak{X}(\pi (M))$: the $\mathfrak{F}(\T M)$-module of differentiable sections of  $\pi^{-1}(T M)$,\\
$ i_{X}$ : the interior product with respect to  $X
\in\mathfrak{X}(M)$,
\par
Elements  of  $\mathfrak{X}(\pi (M))$ will be called
$\pi$-vector fields and will be denoted by barred letters
$\overline{X} $. Tensor fields on $\pi^{-1}(TM)$ will be called
$\pi$-tensor fields. The fundamental $\pi$-vector field is the
$\pi$-vector field $\overline{\eta}$ defined by
$\overline{\eta}(u)=(u,u)$ for all $u\in \T M$. We have the
following short exact sequence of vector bundles:\vspace{-0.1cm}
$$0\longrightarrow
 \pi^{-1}(TM)\stackrel{\gamma}\longrightarrow T(\T M)\stackrel{\rho}\longrightarrow
\pi^{-1}(TM)\longrightarrow 0 ,\vspace{-0.1cm}$$
 the bundle morphisms  $\rho$ and $\gamma$ being defined as usual \cite{r92}.

Let $D$ be  a linear connection (or simply a connection) on the
pullback bundle $\pi^{-1}(TM)$.
 The connection (or the deflection) map associated with $D$ is defined by \vspace{-0.1cm}
$$K:T \T M\longrightarrow \pi^{-1}(TM):X\longmapsto D_X \overline{\eta}.
\vspace{-0.1cm}$$ A tangent vector $X\in T_u (\T M)$ at $u\in \T M$
is horizontal if $K(X)=0$ . The vector space $H_u (\T M)= \{ X \in
T_u (\T M) : K(X)=0 \}$ is the horizontal space at $u$.
A connection $D$ is said to be regular if\,
$T_u (\T M)=V_u (\T M)\oplus H_u (\T M)\,\, \forall u\in \T M $, where $V_u (\T M)$ is the vertical space at $u$.
Let $\beta:=(\rho |_{H(\T M)})^{-1}$, called the horizontal map of the connection
$D$,  then
   \begin{align*}\label{fh1}
    \rho\circ\beta = id_{\pi^{-1} (TM)}, \quad  \quad
       \beta\circ\rho =   id_{H(\T M)} & {\,\,\,\, \text{on}\,\,   H(\T M)}.\vspace{-0.25cm}
\end{align*}
For a regular connection $D$, the horizontal and vertical covariant
derivatives $\stackrel{1}D$ and $\stackrel{2}D$ are defined, for a
vector (1)$\pi$-form $A$, for example, by \vspace{-0.2cm}
 $$ (\stackrel{1}D A)(\overline{X}, \overline{Y}):=
  (D_{\beta \overline{ X}} A)( \overline{Y}), \quad
  (\stackrel{2}D A)( \overline{X},  \overline{Y}):= (D_{\gamma \overline{X}} A)(  \overline{Y}).$$
\par
The horizontal
((h)h-) and mixed ((h)hv-) torsion tensors of the connection $D$ are defined respectively
by \vspace{-0.1cm}
$$Q (\overline{X},\overline{Y}):=\textbf{T}(\beta \overline{X}\beta \overline{Y}),
\, \,\,\,\, T(\overline{X},\overline{Y}):=\textbf{T}(\gamma
\overline{X},\beta \overline{Y}) \quad \forall \,
\overline{X},\overline{Y}\in\mathfrak{X} (\pi (M)),$$
where $\textbf{T}(X,Y)=D_X \rho Y-D_Y\rho X -\rho [X,Y] $ is the (classical)  torsion of the connection $D$\\
%\par
The horizontal (h-), mixed (hv-) and vertical (v-)
curvature tensors of $D$ are defined respectively by
$$R(\overline{X},\overline{Y})\overline{Z}:=\textbf{K}(\beta
\overline{X}\beta \overline{Y})\overline{Z},\quad
P(\overline{X},\overline{Y})\overline{Z}:=\textbf{K}(\beta
\overline{X},\gamma \overline{Y})\overline{Z},\quad
S(\overline{X},\overline{Y})\overline{Z}:=\textbf{K}(\gamma
\overline{X},\gamma \overline{Y})\overline{Z},\vspace{-5pt}$$
where\, $ \textbf{K}(X,Y)\rho Z=-D_X D_Y \rho Z+D_Y D_X \rho Z+D_{[X,Y]}\rho Z$\, is the (classical) curvature tensor of the connection $D$.\\
The cotracted curvatures of $D$ or the (v)h-, (v)hv- and (v)v-torsion tensors  are defined respectively by
$$\widehat{R}(\overline{X},\overline{Y}):={R}(\overline{X},\overline{Y})\overline{\eta},\quad
\widehat{P}(\overline{X},\overline{Y}):={P}(\overline{X},\overline{Y})\overline{\eta},\quad
\widehat{S}(\overline{X},\overline{Y}):={S}(\overline{X},\overline{Y})\overline{\eta}.$$

\begin{thm}{\em\cite{r92}} \label{th.1a} Let $(M,L)$ be a Finsler manifold. There exists a
unique regular connection ${{D}}^{\circ}$ on $\pi^{-1}(TM)$ such
that
\begin{description}
 \item[(a)] $D^{\circ}_{h^{\circ}X}L=0$,
  \item[(b)]   ${{D}}^{\circ}$ is torsion-free\,{\em:} ${\textbf{T}}^{\circ}=0 $,
  \item[(c)]The (v)hv-torsion tensor $\widehat{P^{\circ}}$ of ${D}^{\circ}$ vanishes\,\emph{:}
   $\widehat{P^{\circ}}(\overline{X},\overline{Y})= 0$.
  \end{description}
Such a connection is called Berwald connection associated with $(M,L)$.
\end{thm}
\vspace{-2pt}

Throughout the paper $R^\circ,\, \widehat{R^{\circ}}$ and $H:=i_{\overline{\eta}}\,\widehat{R^{\circ}}$ will denote respectively
the $h$-curvature, the $(v)h$-torsion and the deviation tensor  of Berwald connection $D^\circ$. Moreover, $\stackrel{1}{D^{\circ}}$  and $\stackrel{2}{D^{\circ}}$  will denote respectively the horizontal covariant derivative and the vertical covariant derivative associated with $D^{\circ}$.

%%%%%%%%%%%%%%%%%%%%%%%%%%%%%%%%%%%%%%%%%%%%%%%%%%%%%%% SECTION 2.  %%%%%%%%%%%%%%%%%%%%%%%%%%%%%%%%%%%%%%%%%%%%%%%%%%%

\Section{Finsler Space of Scalar Curvature}

In this section, we establish intrinsically some characterizations of the property of being of scalar curvature.

Let $(M,L)$ be a Finsler manifold and $g$ the Finsler metric defined by $L$. Denote $\ell:=L^{-1}i_{\overline{\eta}}\,g$,\,\,
$\phi(\overline{X}):=\overline{X}-L^{-1}\ell(\overline{X})\overline{\eta}$\, and  \,$\hbar( \overline{X},\overline{Y}):=g(\phi(\overline{ X}),
\overline{ Y})=g(\overline{ X}, \overline{Y})-\ell(\overline{ X}) \ell(\overline{Y})$\,, the
angular metric tensor.

\begin{defn}\label{sca.}{\em{\cite{r86}}} A Finsler manifold $(M,L)$ of dimension $n\geq 3$ is
said to be of scalar curvature  if the deviation tensor $H$ satisfies
$$H(\overline{X})=k L^{2} \phi(\overline{X}), $$
where $k$ is a scalar function on $\T M$, positively homogenous of degree zero in $y$ ($h(0)$)
\footnote{$\omega$ is $h(r)$ in $y$ iff $D^{\circ}_{\gamma \overline{\eta}}\,\omega=r\omega$.},
called scalar curvature \\
In particular, if the scalar curvature $k$ is constant, then $(M,L)$ is called a Finsler manifold of constant curvature.
\end{defn}

\begin{defn}\label{ind.}{\em{\cite{r86}}} An operator $\mathcal{P}$, called the projection operator of indicatrix, is defined as follows:\\
(a) If $\omega$ is a $\pi$-tensor field of type {\em(1,p)}, then %\vspace{-0.3cm}
\begin{equation*}\label{eq.i2}
    (\mathcal{P}\cdot \omega)(\overline{X}_{1},..., \overline{X}_{p}):= \phi(\omega(\phi(\overline{X}_{1}),..., \phi(\overline{X}_{p}))).
\end{equation*}
(b) If $\omega$ is a $\pi$-tensor field  of type {\em(0,p)}, then %\vspace{-0.3cm}
\begin{equation*}\label{eq.i2}
    (\mathcal{P}\cdot \omega)(\overline{X}_{1},..., \overline{X}_{p}):= \omega(\phi(\overline{X}_{1}),..., \phi(\overline{X}_{p})).
    \end{equation*}
(c) In particular, a $\pi$-tensor field $\omega$ is said to be an indicatory tensor
    if\, $\mathcal{P}\cdot \omega=\omega$.
\end{defn}

\begin{rem} \label{rem.1}
The projection $\mathcal{P}$ preserves tensor type. The $\pi$-tensor fields $\phi$ and $\hbar$
are indicatory. Moreover, for any $\pi$-tensor field $\omega$, $\mathcal{P}\cdot \omega$ is indicatory.
\end{rem}

The following result provides some characterizations of Finsler spaces of scalar\linebreak  curvature.
\begin{thm}\label{thm.1} Let $(M,L)$ be a Finsler manifold  of dimension $n\geq 3$. The following assertion are equivalent:
\begin{description}
  \item[(a)] $(M,L)$ is of scalar curvature $k$,
  \item[(b)] The $(v)h$-torsion tensor $\widehat{R^{\circ}}$ satisfies
\begin{equation}\label{(v)h-tors}
  \widehat{R^{\circ}}(\overline{X},\overline{Y})=\mathfrak{A}_{\overline{X},\overline{Y}}
  \{L\phi(\overline{Y})[k\ell(\overline{X})+\frac{1}{3}C(\overline{X})]\}.
\end{equation}
\item[(c)] The $h$-curvature tensor ${R^{\circ}}$ has the form
\begin{eqnarray}\label{h-curv}
   {R^{\circ}}(\overline{X},\overline{Y})\overline{Z}&=& \mathfrak{A}_{\overline{X},\overline{Y}}\{\phi(\overline{Y})[(k\ell(\overline{X})+\frac{1}{3}C(\overline{X}))
    +\frac{1}{3}B(\overline{Z},\overline{X})\nonumber\\
   && + \frac{2}{3}\ell(\overline{X})C(\overline{Z})+k\hbar(\overline{Z},\overline{X})]+
   \frac{1}{3}\ell(\overline{X})C(\overline{Y})\phi(\overline{Z})\nonumber\\
   &&+L^{-1}\hbar(\overline{X},\overline{Z})\overline{\eta}\,
   [k\ell(\overline{Y})+\frac{1}{3}C(\overline{Y})]\},
\end{eqnarray}
\end{description}
where\,\,\,
$\mathfrak{A}_{\overline{X},\overline{Y}}\set{\omega(\overline{X},\overline{Y})}:=
\omega(\overline{X},\overline{Y})-\omega(\overline{Y},\overline{X})$,
\begin{eqnarray}\label{C}
C(\overline{X})&:=&L(\stackrel{2}{D^{\circ}}k)(\overline{X}),\ \    B(\overline{X},\overline{Y}):=L(\mathcal{P}\cdot \stackrel{2}{D^{\circ}}C)(\overline{X}, \overline{Y})
\end{eqnarray}
\end{thm}

To prove this theorem, we need the following three lemmas, which can easily be proved.

\begin{lem}\label{lem.1}
For a Finsler manifold $(M,L)$, we have:
\begin{description}

 \item[{\textbf{(a)}}]
$i_{\overline{\eta}}\,\ell=L, \quad  i_{\overline{\eta}}\,\phi=0, \quad  i_{\overline{\eta}}\,\hbar=0.$

 \item[{\textbf{(b)}}]
$\stackrel{1}{D^{\circ}}L=0, \quad \stackrel{1}{D^{\circ}}\ell=0.$

 \item[{\textbf{(c)}}]
$\stackrel{2}{D^{\circ}}L=\ell, \quad
\stackrel{2}{D^{\circ}}\ell=L^{-1}\hbar.$

\item[{\textbf{(d)}}]
$\stackrel{2}{D^{\circ}}\phi=-L^{-2}\set{\hbar\otimes \overline{\eta}+L\phi
\otimes \ell}.$

 \item[{\textbf{(e)}}]$\mathcal{P}\cdot \ell=0 ,  \quad \mathcal{P}\cdot \hbar=\hbar.$
\end{description}
\end{lem}

\begin{lem}\label{lem.2} The $\pi$-scalar form $C$, defined by (\ref{C}),  has the following properties:
\begin{description}
    \item[{\textbf{(a)}}]
    $i_{\overline{\eta}}\,C=0$

   \item[{\textbf{(b)}}]
   $\mathcal{P}\cdot C=C$ ($C$ is indicatory)

  \item[{\textbf{(c)}}]
  $(\stackrel{2}{D^{\circ}}C)(\overline{\eta}, \overline{X})=0$ ($C$ is $h(0)$),

   \item[{\textbf{(d)}}]
   $(\stackrel{2}{D^{\circ}}C)(\overline{X},\overline{\eta})=-C(\overline{X}).$
\end{description}
\end{lem}

\begin{lem}\label{lem.3} The $\pi$-scalar form $B$, defined by (\ref{C}),  has the following properties:
\begin{description}
    \item[{\textbf{(a)}}]
    $i_{\overline{\eta}}\,B=0$

   \item[{\textbf{(b)}}]
   $\mathcal{P}\cdot B=B$ ($B$ is indicatory)

  \item[{\textbf{(c)}}]
  $(\stackrel{2}{D^{\circ}}B)(\overline{\eta}, \overline{X},\overline{Y})=0$ \,($B$ is $h(0)$),

\item[{\textbf{(d)}}]
  $B(\overline{X},\overline{Y})=L(\stackrel{2}{D^{\circ}}C)(\overline{X}, \overline{Y})+C(\overline{X})\ell(\overline{Y}),$

   \item[{\textbf{(e)}}]$B$ is  symmetric.

    \item[{\textbf{(f)}}]
   $(\stackrel{2}{D^{\circ}}B)(\overline{X},\overline{\eta},\overline{Y})=(\stackrel{2}{D^{\circ}}B)(\overline{X},\overline{Y}, \overline{\eta})=-B(\overline{X}, \overline{Y}).$

\end{description}
\end{lem}

\bigskip

\noindent\textit{\textbf{Proof of Theorem \ref{thm.1}}}:
\par
\vspace{5pt}
$\textbf{(a)}\Rightarrow\textbf{(b)}$: Let $(M,L)$ be a Finsler
manifold of scaler curvature $k$. Then, by Definition \ref{sca.},
the deviation tensor $H$ has the form
\begin{equation}\label{eq.1}
    H(\overline{X})=
  k L^{2} \phi(\overline{X}).
\end{equation}
On the other hand,  from Theorem 4.6 of \cite{r96},  we have
\begin{equation}\label{eq.2}
   \widehat{ R^{\circ}}({\overline{X},\overline{Y}})=\frac{1}{3}\mathfrak{A}_{\overline{X},\overline{Y}}\set{(\stackrel{2}{D^{\circ}}H)(\overline{X},\overline{Y})}.
\end{equation}
From which, together with (\ref{eq.1}) and  Lemma \ref{lem.1}, we
get
\begin{eqnarray*}
  \widehat{R^{\circ}}(\overline{X},\overline{Y} )
  &=&\frac{1}{3}\,\mathfrak{A}_{\overline{X},\overline{Y}}\set{[L^{2}(\stackrel{2}{D^{\circ}}k)\otimes \phi
  +k (\stackrel{2}{D^{\circ}}L^{2})\otimes \phi+k L^{2} (\stackrel{2}{D^{\circ}}\phi)](\overline{X},\overline{Y})}  \\
   &=&\frac{1}{3}\,\emph{}\mathfrak{A}_{\overline{X},\overline{Y}}\set{[LC\otimes \phi
  +2Lk \ell \otimes \phi-k \set{\hbar\otimes \overline{\eta}+L\phi
\otimes \ell}](\overline{X},\overline{Y})}  \\
   &=&\frac{1}{3}\,\mathfrak{A}_{\overline{X},\overline{Y}}\set{[LC\otimes \phi
  +3Lk \ell \otimes \phi-k \hbar\otimes \overline{\eta}](\overline{X},\overline{Y})}.
\end{eqnarray*}
Hence, the result follows.

\par
\vspace{5pt}
$\textbf{(b)}\Rightarrow\textbf{(c)}$: Suppose that the
$(v)h$-torsion tensor $\widehat{R^{\circ}}$ satisfies (\ref{(v)h-tors}):
\begin{equation}\label{eq.3}
    \widehat{R^{\circ}}(\overline{X},\overline{Y})=\mathfrak{A}_{\overline{X},\overline{Y}}\{L\phi(\overline{Y})[k\ell(\overline{X})
    +\frac{1}{3}C(\overline{X})]\}.
\end{equation}
In view of Theorem 4.6 of \cite{r96},  we have
\begin{equation*}
  R^{\circ}({\overline{X},\overline{Y}})\overline{Z}=(\stackrel{2}{D^{\circ}}\widehat{R^{\circ}})(\overline{Z},\overline{X},\overline{Y}).
\end{equation*}
From which, taking into account (\ref{eq.3}) and Lemmas
\ref{lem.1}, \ref{lem.2} and \ref{lem.3}, the result follows.

\par
\vspace{5pt}
$\textbf{(c)}\Rightarrow\textbf{(a)}$: Suppose that the
$h$-curvature tensor  ${R^{\circ}}$ has the form (\ref{h-curv}):
\begin{eqnarray*}
   {R^{\circ}}(\overline{X},\overline{Y})\overline{Z}&=& \mathfrak{A}_{\overline{X},\overline{Y}}\{\phi(\overline{Y})\{[k\ell(\overline{X})+\frac{1}{3}C(\overline{X})]
    +\frac{1}{3}B(\overline{Z},\overline{X})\\
   && + \frac{2}{3}\ell(\overline{X})C(\overline{Z})+k\hbar(\overline{Z},\overline{X})\}+\frac{1}{3}\ell(\overline{X})C(\overline{Y})\phi(\overline{Z})\\
   &&+L^{-1}\hbar(\overline{X},\overline{Z})\overline{\eta}\,
   [k\ell(\overline{Y})+\frac{1}{3}C(\overline{Y})]\},
\end{eqnarray*}
Setting $\overline{X}=\overline{\eta}$ and $\overline{Z}=\overline{\eta}$ into the above equation,
taking into account  Lemmas \ref{lem.1}, \ref{lem.2} and \ref{lem.3}, the result follows.
$\qquad\qquad\qquad\qquad\qquad\quad\qquad\qquad\qquad\qquad\qquad\quad\qquad\qquad\square$

\bigskip
\vspace{5pt}
Let \,$R^\circ(\overline{X},\overline{Y},\overline{Z},\overline{W}):=g(R^\circ(\overline{X},\overline{Y})\overline{Z},\overline{W})$, then we have:
\begin{cor}\label{cor.1} For a Finsler manifold of scalar curvature $k$, the
$h$-curvature tensor  ${R^{\circ}}$ satisfies:
\begin{description}
  \item[(a)]  ${R^{\circ}}(\overline{X},\overline{Y},\overline{Z}, \overline{W})-{R^{\circ}}(\overline{X},\overline{Y},\overline{W}, \overline{Z})= \mathfrak{A}_{\overline{X},\overline{Y}}\{\hbar(\overline{Z}, \overline{X})N(\overline{W}, \overline{Y})
  +\hbar(\overline{W}, \overline{Y})N(\overline{Z}, \overline{X})\}.$

  \item[(b)] ${R^{\circ}}(\overline{X},\overline{Y},\overline{Z}, \overline{W})+{R^{\circ}}(\overline{X},\overline{Y},\overline{W}, \overline{Z})=
   \mathfrak{A}_{\overline{X},\overline{Y}}\{\hbar(\overline{W}, \overline{Y})F(\overline{Z}, \overline{X})+\hbar(\overline{Z}, \overline{Y})F(\overline{W},    \overline{X})$\\
  ${\qquad\qquad\qquad\qquad\qquad\qquad\qquad\qquad }+\hbar(\overline{W}, \overline{Z})F(\overline{Y}, \overline{X})\}.$
\end{description}
where $N$ and $F$ are the $\pi$-tensor fields of type $(0,2)$
defined respectively by
\begin{eqnarray*}
   N(\overline{X},\overline{Y}):&=&k\set{g(\overline{X},\overline{Y})+\ell(\overline{X})\ell(\overline{Y})}\nonumber\\
   &&+\frac{1}{3}\set{B(\overline{X},\overline{Y})+2\ell(\overline{X})C(\overline{Y})+2C(\overline{X})\ell(\overline{Y})},\label{eq.4}\\
   F(\overline{X},\overline{Y})&:=&\frac{1}{3}\set{B(\overline{X},\overline{Y})+2C(\overline{X})\ell(\overline{Y})}\label{eq.5}.
   \end{eqnarray*}
\end{cor}

\bigskip

We end this section by the following result.
\begin{prop}\label{lem.4} For a Finsler manifold of scalar curvature, $\mathcal{P}\cdot R^{\circ}$ vanishes  if and only if\,
$\mathcal{P}\cdot N$ vanishes, where $N$  is the $\pi$-form defined
by Corollary \ref{cor.1}.
\end{prop}

\begin{proof} Let $(M,L)$ be a Finsler manifold of scalar curvature. Then, by Theorem \ref{thm.1}, we have
\begin{equation}\label{eq.23}
    ({\mathcal{P} \cdot R^{\circ}})(\overline{X},\overline{Y},\overline{Z}, \overline{W})=\mathfrak{A}_{\overline{X},    \overline{Y}}\{\hbar(\overline{Y},\overline{W})[\frac{1}{3}B(\overline{Z},\overline{X})+k\hbar(\overline{Z},\overline{X})]\}.
\end{equation}
On the other hand, by Corollary \ref{cor.1}, we obtain
\begin{equation*}
    (\mathcal{P} \cdot N)(\overline{X},\overline{Y})=\frac{1}{3}B(\overline{Z},\overline{X})+k\hbar(\overline{Z},\overline{X}).
\end{equation*}
From which together with (\ref{eq.23}), we get
\begin{equation}\label{eq.24}
    (\mathcal{P} \cdot R^{\circ})(\overline{X},\overline{Y},\overline{Z})=\mathfrak{A}_{\overline{X},\overline{Y}}\{\phi(\overline{Y})(\mathcal{P} \cdot N)(\overline{Z}, \overline{X})\}.
\end{equation}
It is clear, from (\ref{eq.24}), that if $\mathcal{P}\cdot N=0$, then $\mathcal{P}\cdot R^{\circ}=0.$\\
Conversely, if $\mathcal{P}\cdot R^{\circ}=0$, it follows again from (\ref{eq.24}) that
$$\mathfrak{A}_{\overline{X},\overline{Y}}\{\phi(\overline{Y})(\mathcal{P} \cdot N)(\overline{Z},\overline{X})\}=0.$$
Taking the contracted trace with respect to $\overline{Y}$, the above
relation reduces to
$$(n-2)(\mathcal{P} \cdot N)(\overline{Z}, \overline{X})=0.$$
Consequently, as $n\geq3$, $\mathcal{P}\cdot N$ vanishes.
\end{proof}
%\vspace{-12pt}

%%%%%%%%%%%%%%%%%%%%%%%%%%%%%%%%%%%%%%%%%%%%%%%%%%%%%%%%% Section 3 %%%%%%%%%%%%%%%%%%%%%%%%%%%%%%%%%%$$$$$$$$$$$$$$$

\Section{Finsler Space of Constant Curvature}

In this section, we investigate intrinsically necessary and sufficient
conditions under which a Finsler manifold of scalar curvature reduces
to a Finsler manifold of constant curvature.

\vspace{4pt}
The following lemma is useful for subsequent
use.%\vspace{-0.2cm}

\begin{lem}\label{pro.1} For a Finsler manifold $(M,L)$  of constant curvature, we have
$$\mathfrak{A}_{\overline{X},\overline{Y}}\set{A(\overline{X},\overline{Y},\overline{Z})+C(\overline{X})\hbar(\overline{Y}, \overline{Z})}=0,$$
where $A$ is the $\pi$-tensor field of type (0,3) defined
by
\begin{eqnarray*}
A(\overline{X},\overline{Y},\overline{Z})&:=&L(\mathcal{P} \cdot
\stackrel{2}{D^{\circ}}B)(\overline{X}, \overline{Y}, \overline{Z}).
\end{eqnarray*}
and $C$, $B$ are the $\pi$-tensor fields defined by (\ref{C}).
\end{lem}

\begin{proof}
We have
\begin{eqnarray*}
  A(\overline{X},\overline{Y},\overline{Z}) &=& L(\mathcal{P} \cdot
\stackrel{2}{D^{\circ}}B)(\overline{X}, \overline{Y}, \overline{Z}). \\
   &=& L(\stackrel{2}{D^{\circ}}B)(\phi(\overline{X}), \phi(\overline{Y}), \phi(\overline{Z})). \\
   &=&L(\stackrel{2}{D^{\circ}}B)(\overline{X}, \overline{Y}, \overline{Z})
   -\ell(\overline{Z})(\stackrel{2}{D^{\circ}}B)(\overline{X}, \overline{Y}, \overline{\eta})\\
   &&   -\ell(\overline{Y})(\stackrel{2}{D^{\circ}}B)(\overline{X}, \overline{\eta}, \overline{Z})
   +L^{-1}\ell(\overline{Y})\ell(\overline{Z}) (\stackrel{2}{D^{\circ}}B)(\overline{X}, \overline{\eta},\overline{\eta})\\
&&  -\ell(\overline{X})(\stackrel{2}{D^{\circ}}B)(\overline{\eta}, \overline{Y}, \overline{Z})
   +L^{-1}\ell(\overline{X})\ell(\overline{Z}) (\stackrel{2}{D^{\circ}}B)(\overline{\eta}, \overline{Y}, \overline{\eta})\\
   && +L^{-1}\ell(\overline{X})\ell(\overline{Y}) (\stackrel{2}{D^{\circ}}B)(\overline{\eta}, \overline{\eta},    \overline{Z})-L^{-2}\ell(\overline{X})\ell(\overline{Y})\ell(\overline{Z}) (\stackrel{2}{D^{\circ}}B)(\overline{\eta}, \overline{\eta},
   \overline{\eta}).
\end{eqnarray*}
In view Lemmas \ref{lem.1} and \ref{lem.3}, the above Equation
reduces to
\begin{equation}\label{eq.5}
  A(\overline{X},\overline{Y},\overline{Z}) = L(\stackrel{2}{D^{\circ}}B)(\overline{X}, \overline{Y}, \overline{Z})+
  \ell(\overline{Z})B(\overline{X}, \overline{Y})+\ell(\overline{Y})B(\overline{X}, \overline{Z}).
\end{equation}
On the other hand, from (\ref{C}) and Lemmas \ref{lem.1}, \ref{lem.2} and \ref{lem.3}, we have
\begin{eqnarray*}
    (\stackrel{2}{D^{\circ}}B)(\overline{X}, \overline{Y}, \overline{Z})&=&
   \ell(\overline{X})(\stackrel{2}{D^{\circ}}C)(\overline{Y}, \overline{Z})
   +L(\stackrel{2}{D^{\circ}}\stackrel{2}{D^{\circ}}C)(\overline{X}, \overline{Y},\overline{Z})\nonumber \\
  &&+L^{-1}\hbar(\overline{X},\overline{Z})C(\overline{Y})+\ell(\overline{Z})(\stackrel{2}{D^{\circ}}C)(\overline{X}, \overline{Y})\nonumber\\
&=&\ell(\overline{X})(\stackrel{2}{D^{\circ}}C)(\overline{Y}, \overline{Z})
   +L(\stackrel{2}{D^{\circ}}\stackrel{2}{D^{\circ}}C)(\overline{X}, \overline{Y},\overline{Z})\nonumber \\
  &&+L^{-1}\hbar(\overline{X},\overline{Z})C(\overline{Y})+L^{-1}\ell(\overline{Z})\ell(\overline{X})C(\overline{Y})\nonumber\\
  &&+L\ell(\overline{Z})(\stackrel{2}{D^{\circ}}\stackrel{2}{D^{\circ}}k)(\overline{X}, \overline{Y})\nonumber\\
&=&L^{-1}\ell(\overline{X})B(\overline{Y}, \overline{Z})
   +L(\stackrel{2}{D^{\circ}}\stackrel{2}{D^{\circ}}C)(\overline{X}, \overline{Y},\overline{Z})\\
  &&+L^{-1}\hbar(\overline{X},\overline{Z})C(\overline{Y})+L\ell(\overline{Z})(\stackrel{2}{D^{\circ}}\stackrel{2}{D^{\circ}}k)(\overline{X}, \overline{Y})\nonumber.
\end{eqnarray*}
From which, together (\ref{eq.5}), we obtain
\begin{eqnarray}
 A(\overline{X},\overline{Y},\overline{Z})&=& \mathfrak{S}_{\overline{X}, \overline{Y},\overline{Z}}\set{\ell(\overline{X})B(\overline{Y}, \overline{Z})}
 + \hbar(\overline{X},\overline{Z})C(\overline{Y})\nonumber\\
 && +L^{2}\set{(\stackrel{2}{D^{\circ}}\stackrel{2}{D^{\circ}}C)(\overline{X}, \overline{Y},\overline{Z})
 +\ell(\overline{Z})(\stackrel{2}{D^{\circ}}\stackrel{2}{D^{\circ}}k)(\overline{X}, \overline{Y})}\label{eq.6}.
 \end{eqnarray}
On the other hand, for every $(1)\pi$-form $\omega$, one can show that
\begin{equation}\label{eq.7}
    (\stackrel{2}{D^{\circ}}\stackrel{2}{D^{\circ}}\omega)(\overline{X}, \overline{Y},\overline{Z})-(\stackrel{2}{D^{\circ}}\stackrel{2}{D^{\circ}}\omega)(\overline{Y}, \overline{X},\overline{Z})=0
\end{equation}
Then the result follows from (\ref{eq.6}) and (\ref{eq.7}).
\end{proof}

\begin{thm}\label{thm.2} A Finsler manifold $(M,L)$ of scalar curvature $k$ reduces to
 a Finsler manifold of constant curvature $k$ if and only if the $\pi$-scalar form
$C=L\stackrel{2}{D^{\circ}}k$ vanishes.
\end{thm}

\begin{proof} Firstly, suppose that  $(M,L)$ is Finsler manifold of scalar curvature $k$. If $(M,L)$ reduces to a Finsler
manifold of constant curvature $k$,  then the $\pi$-scalar form $C$ vanishes immediately.
\par
Conversely, suppose that $(M,L)$ is a Finsler manifold of scalar curvature $k$ such that  the
$\pi$-scalar form $C$ vanishes. Hence
\begin{equation}\label{eq.8}
    \stackrel{2}{D^{\circ}}k=0.
\end{equation}
By (\ref{(v)h-tors}), together with $C=0$, we obtain
 \begin{equation}\label{eq.9}
    \widehat{R^{\circ}}(\overline{X},\overline{Y})=kL\set{\ell(\overline{X})\overline{Y}-\ell(\overline{Y})\overline{X}}.
 \end{equation}
On the other hand, we have \cite{r96}:
\begin{equation*}\label{eq.10}
   \mathfrak{S}_{\overline{X},\overline{Y},\overline{Z}}\,
\{(D^{\circ}_{\beta \overline{X}}R^{\circ})(\overline{Y},
\overline{Z},\overline{W})+P^{\circ}(\widehat{R^{\circ}}(\overline{X},\overline{Y}),
\overline{Z})\overline{W}\}=0.
\end{equation*}
 From which, noting that the $(v)hv$-torsion $\widehat{P^{\circ}}$ vanishes \cite{r96}, it follows that
\begin{equation}\label{eq.11}
   \mathfrak{S}_{\overline{X},\overline{Y},\overline{Z}}\,
(D^{\circ}_{\beta \overline{X}}\widehat{R^{\circ}})(\overline{Y},
\overline{Z})=0.
\end{equation}
Now, from (\ref{eq.9}) and (\ref{eq.11}), using (\ref{eq.1}) and $D^{\circ}_{\beta
\overline{X}}\ell=0$, we get
\begin{eqnarray*}
  && L(D^{\circ}_{\beta
\overline{X}}k)(\ell(\overline{Y})\overline{Z}-\ell(\overline{Z})\overline{Y})
+  L(D^{\circ}_{\beta
\overline{Y}}k)(\ell(\overline{Z})\overline{X}-\ell(\overline{X})\overline{Z}) \\
   &&+L(D^{\circ}_{\beta
\overline{Z}}k)(\ell(\overline{X})\overline{Y}-\ell(\overline{Y})\overline{X})=0.
\end{eqnarray*}
Setting $\overline{Z}=\overline{\eta}$ into the above equation,
noting that $\ell(\overline{\eta})=L$ (Lemma \ref{lem.1}), we obtain
\begin{eqnarray*}
  && L({D^{\circ}}_{\beta
\overline{X}}k)(\ell(\overline{Y})\overline{\eta}-L\overline{Y})
+  L(D^{\circ}_{\beta
\overline{Y}}k)(L\overline{X}-\ell(\overline{X})\overline{\eta}) \\
   &&+L(D^{\circ}_{\beta
\overline{\eta}}k)(\ell(\overline{X})\overline{Y}-\ell(\overline{Y})\overline{X})=0.
\end{eqnarray*}
Taking the trace of both sides with respect to $\overline{Y}$, it
follows that
\begin{equation}\label{eq.12}
{D^{\circ}}_{\beta \overline{X}}k=L^{-1}({D^{\circ}}_{\beta
\overline{\eta}}k)\ell(\overline{X}).
\end{equation}
Applying the $v$-covariant derivative with respect to $\overline{Y}$
on both sides of (\ref{eq.12}), yields
\begin{equation*}
  \ell(\overline{Y})
  D^{\circ}_{\beta\overline{X}}k+L(\stackrel{2}{D^{\circ}}\stackrel{1}{D^{\circ}}
  k)(\overline{X},\overline{Y})=
  L^{-1}\hbar(\overline{X},\overline{Y})(D^{\circ}_{\beta\overline{\eta}}k)+\ell(\overline{X})
  (\stackrel{2}{D^{\circ}}\stackrel{1}{D^{\circ}} k)(\overline{\eta},\overline{Y}).
\end{equation*}
From (\ref{eq.8}), noting that $(\stackrel{2}{D^{\circ}}\stackrel{1}{D^{\circ}}
  k)(\overline{X},\overline{Y})=(\stackrel{1}{D^{\circ}}\stackrel{2}{D^{\circ}}
  k)(\overline{Y},\overline{X})$, the above relation reduces to (provided that $n\geq3$)
\begin{equation*}
  \ell(\overline{Y})
  D^{\circ}_{\beta\overline{X}}k=
  L^{-1}\hbar(\overline{X},\overline{Y})(D^{\circ}_{\beta\overline{\eta}}k).
\end{equation*}
Setting $\overline{Y}=\overline{\eta}$ into the above equation,
noting that $\ell(\overline{\eta})=L$ and
$\hbar(.,\overline{\eta})=0$, it follows that
$D^{\circ}_{\beta\overline{X}}k=0$. Consequently,
\begin{equation}\label{eq.13}
\stackrel{1}{D^{\circ}}k=0  .
\end{equation}
Now, (\ref{eq.8}) and (\ref{eq.13}) imply  that $k$ is a
constant.
\end{proof}

\begin{thm}\label{thm.3} A Finsler manifold $(M,L)$ of scalar curvature $k$ reduces to
 a Finsler manifold of constant curvature $k$ if and only if the $\pi$-scalar form
$B=L(\mathcal{P}\cdot\stackrel{2}{D^{\circ}}C)$ vanishes.
\end{thm}

\begin{proof} Let  $(M,L)$ be a Finsler manifold of scalar curvature $k$.

If $(M,L)$ reduces to a Finsler manifold of constant curvature $k$, then, by Theorem \ref{thm.2}, the $\pi$-scalar form $C$
vanishes. Consequently, the $\pi$-scalar form $B$ vanishes.

Conversely, suppose that $(M,L)$ has the property that the
$\pi$-scalar form $B$ vanishes. Hence, the $\pi$-scalar form $A$ of Lemma \ref{pro.1}
vanishes. Consequently by Lemma \ref{pro.1}, we have
\begin{equation*}
C(\overline{X})\phi(\overline{Y})-C(\overline{Y})\phi(\overline{X})=0
\end{equation*}
Taking the trace of both sides of the above equation with respect to
$\overline{Y}$, noting that $Tr(\phi)=n-1$ \cite{r86}, it follows
that
\begin{equation*}
(n-2)C(\overline{X})=0.
\end{equation*}
From which, the $\pi$-scalar form $C$ vanishes as $n\geq3$. Consequently, by Theorem \ref{thm.2}, $(M,L)$ is of constant
curvature $k$.
\end{proof}

\begin{thm}\label{thm.4} A Finsler manifold $(M,L)$ of scalar curvature $k$ reduces to
 a Finsler manifold of constant curvature $k$ if and only if the $\pi$-scalar form
$A=L(\mathcal{P}\cdot\stackrel{2}{D^{\circ}}B)$ vanishes.
\end{thm}

\begin{proof} The proof is similar to that of the above theorem.
\end{proof}

Summing up, we have.
\begin{thm}\label{thm.5} Let $(M,L)$ be a Finsler manifold of scalar curvature $k$.
 The following assertion are equivalent:
\begin{description}
\item[(a)] $(M,L)$ is of constant curvature $k$,
\item[(b)] The $(1)\pi$-scalar form $C=L\stackrel{2}{D^{\circ}}k$ vanishes,
\item[(c)] The $(2)\pi$-scalar form $B=L(\mathcal{P}\cdot\stackrel{2}{D^{\circ}}C)$ vanishes,
\item[(d)] The $(3)\pi$-scalar form $A=L(\mathcal{P}\cdot\stackrel{2}{D^{\circ}}B)$ vanishes.
\end{description}
\end{thm}

\begin{cor}\label{pro.5} A Finsler manifold of scalar curvature is of constant curvature if and only
if $\mathcal{P}\cdot F=0$, where $F$ is the $\pi$-form defined by
Corollary \ref{cor.1}.
\end{cor}

\begin{proof} The proof follows from the identity
\begin{equation*}\label{eq.22}
\mathcal{P}\cdot F(\overline{X},\overline{Y})=\frac{1}{3}B(\overline{X},\overline{Y}).
\end{equation*}
which can easily be proved.
\end{proof}

%%%%%%%%%%%%%%%%%%%%%%%%%%%%%%%%%%%%%%%%%%%%%%%%%%%%%% References %%%%%%%%%%%%%%%%%%%%%%%%%%%%%%%%%%%%%

\vspace{-1cm}
\providecommand{\bysame}{\leavevmode\hbox to3em{\hrulefill}\thinspace}
\providecommand{\MR}{\relax\ifhmode\unskip\space\fi MR }
\providecommand{\MRhref}[2]{%
  \href{http://www.ams.org/mathscinet-getitem?mr=#1}{#2}
}
\providecommand{\href}[2]{#2}

\end{document}